\documentclass[reqno,11pt]{amsart}
\usepackage[utf8]{inputenc}

\usepackage{graphicx}
\usepackage{subcaption}
\usepackage[usenames,dvipsnames]{pstricks}
\usepackage{epsfig}
\usepackage{pst-grad} 
\usepackage{pst-plot} 
\usepackage[space]{grffile} 
\usepackage{etoolbox} 
\makeatletter 
\patchcmd\Gread@eps{\@inputcheck#1 }{\@inputcheck"#1"\relax}{}{}
\makeatother
\usepackage{comment}
\usepackage{mathtools}
\usepackage{amsmath}
\usepackage{amssymb}
\usepackage{amsthm}
\usepackage{ifthen}
\usepackage{color}
\usepackage[UKenglish]{babel}
\usepackage{polski}

\newcommand{\intav}[1]{\mathchoice {\mathop{\vrule width 6pt height 3 pt depth  -2.5pt
\kern -8pt \intop}\nolimits_{\kern -6pt#1}} {\mathop{\vrule width
5pt height 3  pt depth -2.6pt \kern -6pt \intop}\nolimits_{#1}}
{\mathop{\vrule width 5pt height 3 pt depth -2.6pt \kern -6pt
\intop}\nolimits_{#1}} {\mathop{\vrule width 5pt height 3 pt depth
-2.6pt \kern -6pt \intop}\nolimits_{#1}}}

\def\polhk#1{\setbox0=\hbox{#1}{\ooalign{\hidewidth\lower1.5ex\hbox{`}\hidewidth\crcr\unhbox0}}}

 \newcommand{\Rr}{\mathbb R}

\newcommand{\Nn}{\mathbb N}

\newcommand{\entre}{\setminus}

\newtheorem{Theorem}{Theorem}

\newtheorem{Lemma}{Lemma}
\newtheorem{Corollary}{Corollary}
\newtheorem{Proposition}{Proposition}
\newtheorem{assump}{}

\theoremstyle{definition}

\theoremstyle{remark}
\newtheorem{Remark}{Remark}

\usepackage{setspace}
\title[Regularity of degenerate problems]{A degenerate fully nonlinear free transmission problem with variable exponents}

\author[D. Jesus]{David Jesus}
\address{University of Coimbra, CMUC, Department of Mathematics, 3001-501 Coimbra, Portugal}{}
\email{djbj@mat.uc.pt}

\date{\today}

\begin{document}



\begin{abstract}
We study degenerate fully nonlinear free transmission problems, where the degeneracy rate varies in the domain. We prove optimal pointwise regularity depending on the degeneracy rate. Our arguments consist of perturbation methods, relating our problem to a homogeneous, fully nonlinear, uniformly elliptic equation.
\end{abstract}

\maketitle

\section{Introduction}

We consider a free transmission problem
\begin{align}\label{Equation_main_form}
    |Du|^{\beta(x,u,Du)}F\left(D^2 u\right)=f(x)\qquad \mbox{ in } B_1,
\end{align}
where $\beta\geq 0$, $F$ is uniformly elliptic and $f$ is bounded and continuous. The model \eqref{Equation_main_form} accounts for a diffusion process degenerating as a variable power of the gradient.

Transmission problems appear frequently in various fields of physics and biology. They model phenomena which follow different laws in separate subsets of the domain. Typical examples consist of studying mathematical models in composite materials. For a description of these problems, we suggest the readers to the reference \cite{Borsuk2010}. When these subsets depend on the solution itself, these equations become \textit{free} transmission problems.

Problems of the form 
\begin{align*}
    \left|Du\right|^\beta F\left(D^2u\right)=f
\end{align*}
belong to a larger class of equations studied in a series of papers by Birindelli and Demengel, starting with the singular case in \cite{Birindelli-Demengel2004}. The degenerate case was also considered in \cite{Birindelli-Demengel2007,Birindelli-Demengel2009}. An important development concerning higher regularity for degenerate fully nonlinear equations was put forward in  \cite{Imbert-Silvestre2013}. In that paper, the authors obtain local $C^{1,\alpha}$ regularity, for
\begin{align*}
    \alpha\in (0,\alpha_0) \quad \mbox{ and }\quad \alpha\leq\frac{1}{1+\beta},
\end{align*}
with $\alpha_0$ corresponding to the $C^{1,\alpha_0}$ regularity of the homogeneous equation $F\left(D^2u\right)=0$. In \cite[Lemma 6]{Imbert-Silvestre2013}, the authors provide a connection between the homogeneous degenerate equation and the corresponding homogeneous uniformly elliptic equation. This step unlocks a higher regularity class which they access via a tangential path.

The methods introduced in \cite{Imbert-Silvestre2013} resonated, launching new perspectives in the theory of degenerate fully nonlinear equations. In \cite{Bronzi-Pimentel-Rampasso-Teixeira2020}, the authors consider the equation
\begin{align*}
    \left|Du\right|^{\beta(x)}F\left(D^2u\right)=f(x)
\end{align*}
where $\beta$ is allowed to change sign, and obtain local $C^{1,\alpha}$ regularity, where
\begin{align*}
    \alpha\in (0,\alpha_0)\quad \mbox{ and }\quad \alpha\leq\frac{1}{1+\left\|\beta_+\right\|_\infty+\left\|\beta_-\right\|_\infty},
\end{align*}
with $\beta_+$ and $\beta_-$ corresponding to the positive and negative parts of $\beta$, respectively. The estimates obtained in \cite{Bronzi-Pimentel-Rampasso-Teixeira2020} are independent of the continuity modulus of $\beta$.

In \cite{Huaroto-Pimentel-Rampasso-Swiech2020}, the authors consider a degeneracy law depending on the sign of the solution. They study the equation
\begin{align*}
    |Du|^{\beta_+\chi\{u>0\}+\beta_-\chi\{u<0\}}F\left(D^2 u\right)=f(x),
\end{align*}
which has constant degeneracy rates at each of the phases $\{\pm u>0\}$, but has a discontinuity across the free boundary $\partial \{u=0\}$. They obtain local $C^{1,\alpha}$ regularity, for
\begin{align*}
    \alpha\in (0,\alpha_0)\quad \mbox{ and }\quad \alpha\leq\frac{1}{1+\max\left\{\beta_-,\beta_+\right\}}.
\end{align*}
The authors also establish existence of solutions via Perron's method.

Finally, we mention the recent paper \cite{Filippis2020} where the author considered the following equation
\begin{align*}
    \begin{cases}
    \,\left[ |Du|^{p_u(x)} +a(x)\chi_{\{u>0\}}|Du|^q+b(x)\chi_{\{u<0\}}|Du|^s \right]F(D^2u)=f, \, &\mbox{ in } \Omega,\\
    \,u=g \, &\mbox{ on } \partial \Omega,
    \end{cases}
\end{align*}
where $p_u(x)=p^+ \chi_{\{u>0\}}+p^-\chi_{\{u<0\}}$. In this setting, they prove existence and uniqueness of solutions and obtain local $C^{1,\alpha}$ regularity.

The present paper consists of two main results. First, we obtain a local regularity result under very general assumptions on $\beta(x,u,Du)$ which generalizes the results mentioned above. Indeed, we  only require $\beta:B_1\times \Rr\times \Rr^d$ to be well-defined in its domain and bounded from above and below. The first main result in this paper is the following.

 \begin{Theorem}[Local $C^{1,\alpha}$ regularity]\label{Theorem_local_regularity}
 
 Let $u\in C(B_1)$ be a viscosity solution to \eqref{Equation_main_form}.  Assume that $0<\beta_m\leq \beta(x,t,p)\leq\beta_M$ for fixed $\beta_m, \beta_M$;
 assume also that $F$ is uniformly $(\lambda,\Lambda)$-elliptic, $F(0) = 0$ and $f$ is continuous and bounded in $B_1$. Let finally $\alpha_0\in(0,1)$ be given in Remark \ref{Remark_Krylov_Safanov} below. 
  
 Then, there exist $\alpha>0$ and $C > 0$ such that any viscosity solution $u$ of \eqref{Equation_main_form} is in $C^{1,\alpha}(B_{1/2})$ and
\begin{align*}
    [u]_{C^{1,\alpha}(B_{1/2})}\leq C\left(\left\|u\right\|_{L^\infty(B_1)}+\left\|f\right\|_{L^\infty(B_1)}\right),
\end{align*}
where
\begin{align*}
    \alpha=\min\left\{\alpha_0^-,\frac{1}{1+\beta_M}\right\},
\end{align*}
    and $C=C(\lambda, \Lambda, d, \beta_m, \beta_M)$.

\end{Theorem}
We now state some remarks concerning this result.
\begin{Remark}
Theorem \ref{Theorem_local_regularity} includes the following examples
\begin{itemize}
    \item $\beta(x,u,Du)=\beta(x)\chi_{G(u)}$, where $G(u)=B_1\setminus \{u=|Du|=0\}$;
    \item $\beta(x,u,Du)=\theta(|Du|)$, where $\theta(t)\to 2$ as $t\to 0$ and $\theta(t)\to 1$ as $t\to \infty$. This equation was considered for the first time in \cite{Blomgren-Chan-Mulet-Vesse-Wan1999}.
\end{itemize}
\end{Remark}

\begin{Remark}
The regularity class is interpreted in the following sense: If $\frac{1}{1+\beta_M}<\alpha_0$, then solutions are $C^{1,\alpha}(B_{1/2})$ with $\alpha=\frac{1}{1+\beta_M}$; if, alternatively, $\alpha_0\leq \frac{1}{1+\beta_M}$, then solutions are $C^{1,\alpha}(B_{1/2})$ for every $\alpha<\alpha_0$.
\end{Remark}

\begin{Remark}\label{Remark_Krylov_Safanov}
By the classical Krylov-Safanov and Trudinger theory, every viscosity solution of
\begin{align*}
    F(D^2u)=0,\quad\mbox{ in } B_1,
\end{align*}
belongs to $C^{1,\alpha_0}(B_{1/2})$ for a universal $\alpha_0\in(0,1)$ (see for example \cite[Chapter 5]{Caffarelli-Cabre1995}).
\end{Remark}

As one can see from the previous result, there is an intrinsic dependence between the regularity obtained and the degeneracy rate. Hence, if this rate is variable over the domain, it seems natural to obtain regularity results which also vary over the domain. This idea, put forward in Lemma \ref{Lemma_geometric_iterations_pointwise} which corresponds to the geometric iterations, is the novelty in the current paper. By considering variable exponents in each iteration, we are able to better capture the pointwise degenerate behaviour of the equation. 

To obtain this improved regularity, we consider the following explicit expression for the exponent. Let $G_i(u,Du)\subset B_1$, $i=1, ..., N$ be disjoint sets which depend on the solution $u$ and its gradient $Du$, and define $G_0(u,Du):=B_1\setminus \bigcup_{i=0}^N G_i(u,Du)$. Assume the exponent $\beta$ has the form
\begin{align}\label{Equation_exponent_expression}
    \beta(x,u,Du)=\sum_{i=0}^N \beta_i(x)\chi_{G_i(u,Du)}.
\end{align}
An example to keep in mind is the following. Let $N=2$, \linebreak $G_1(u,Du)=\{u>0\}$, $G_2(u,Du)=\{u<0\}$ and $G_0(u,Du)=\{u=0\}$. If $\beta_0=0$, $\beta_1$ and $\beta_2$ are constants, then we recover the result from \cite{Huaroto-Pimentel-Rampasso-Swiech2020}. Our result is thus more refined, not only in the sense that it includes a much broader class of degeneracies, but also because we obtain an improved pointwise regularity. The second main result in this paper is the following.

 \begin{Theorem}[Pointwise $C^{1,\alpha}$ regularity]\label{Theorem_pointwise_regularity}
 
 Let $u\in C(B_1)$ be a viscosity solution to \eqref{Equation_main_form}.  Assume that $\beta$ is given by \eqref{Equation_exponent_expression} with $\beta_i(\cdot)\in[\beta_m,\beta_M]$ for fixed $0\leq\beta_m\leq\beta_M$ and have modulus of continuity $\omega$ satisfying
  \begin{align}\label{Equation_limsup_continuity}
    \limsup_{t\to 0}\ln\left(\frac{1}{t}\right)\omega(t)=0;
\end{align}
   assume also that $F$ is uniformly $(\lambda,\Lambda)$-elliptic, $F(0) = 0$ and $f$ is continuous and bounded in $B_1$. Let finally $\alpha_0\in(0,1)$ be given in Remark \ref{Remark_Krylov_Safanov}. 
  
  Then for every $x_0\in B_{1/2}$, there exist $\alpha>0$ and $C > 0$ such that any viscosity solution $u$ of \eqref{Equation_main_form} is pointwise $C^{1,\alpha}(x_0)$ and
\begin{align*}
    [u]_{C^{1,\alpha}(x_0)}\leq C\left(\left\|u\right\|_{L^\infty(B_1)}+\left\|f\right\|_{L^\infty(B_1)}\right),
\end{align*}
where
\begin{align*}
    \alpha=\min_{i=0, ..., N}\left\{\alpha_0^-,\frac{1}{1+\beta_i(x_0)}\right\},
\end{align*}
and $C=C(\lambda, \Lambda, d, \beta_m, \beta_M,\omega)$.

\end{Theorem}

Assumptions of the type \eqref{Equation_limsup_continuity} are typical when obtaining higher regularity of solutions to equations with variable exponents. For example, in \cite{Acerbi-Mingione2001} the authors are able to prove improved regularity to a class of variational problems with variable exponents, under the assumption above.

Pointwise regularity has been the subject of various papers, see for example \cite{Caffarelli-Cabre1995} and \cite{Lian-Wang-Zhang2020}. These are useful when a certain property is not verified locally but instead only at a point. Obtaining such a result as in Theorem \ref{Theorem_pointwise_regularity} instead of a local regularity result as in Theorem \ref{Theorem_local_regularity}, comes with a cost, since we must assume stronger uniform continuity of the functions $\beta_i$. However, more information is gathered. For example, consider $N=0$,
\begin{align*}
    \beta_0(x)=1000e^{-\frac{1}{2}\left|1000x\right|^2}
\end{align*}
and assume $F$ is convex, so that $\alpha_0=1$ (see \cite[Chapter 6]{Caffarelli-Cabre1995}). Then a local result would yield $C^{1,\alpha}$ regularity, with $\alpha=\frac{1}{1001}$. The problem with this result is that $\beta_0\approx0$ except in a small neighborhood of $0$. On the other hand, Theorem \ref{Theorem_pointwise_regularity} would immediately yield $C^{1,\alpha}$ regularity with $\alpha\approx 1$ for points away from the origin.

Another advantage of having such a sharp pointwise regularity comes when studying the free boundary of the problem, where a finer analysis is required.

The remainder of the paper is organized as follows. Section 2 introduces the assumptions to hold throughout the paper, some basic notation and a characterization of H\"older spaces. We also obtain a simple proof for Theorem \ref{Theorem_local_regularity}. In Section 3, we simplify the equation, rewriting it as viscosity inequalities and removing the dependence of the exponents on the solution. We then obtain an important smallness assumption, which provides a tangential path between our equation and the homogeneous one. H\"older continuity of a perturbed equation is the topic of Section 4. In Section 5 we derive an approximation lemma. Finally, Section 6 consists of the geometric iterations with variable exponents which combined with the characterization of H\"older spaces put forward in Section 2, provides the improved, pointwise H\"older continuity of the gradient.

\section{Preliminary material and main assumptions}

In this introductory section, we present some basic results that will be instrumental in the sequel and detail our main assumptions. We start with some notation. 

For $r>0$, we call $B_r(x)$ the ball in $\Rr^d$ centered around $x$ and with radius $r$. $B_r$ denotes $B_r(0)$. The space of symmetric $d\times d$ real matrices is denoted by $\mathcal{S}(d)$. We say $u\in C^{1,\alpha}(\Omega)$ if $u\in C^1(\Omega)$ and
\begin{align*}
    [u]_{C^{1,\alpha}(\Omega)}:=\sup_{x,\,y\in \Omega}\frac{|Du(x)-Du(y)|}{|x-y|^{\alpha}}<\infty.
\end{align*}
Similarly, we say $u\in C^{1,\alpha}(x_0)$ if $u\in C^1$ in a neighborhood of $x_0$ and
\begin{align*}
    [u]_{C^{1,\alpha}(x_0)}:=\sup_{r>0,\,y\in B_r(x_0)}\frac{|Du(y)-Du(x_0)|}{|y-x_0|^{\alpha}}<\infty.
\end{align*}

Next, we introduce the uniform ellipticity assumption, assumed to hold throughout the paper.
\begin{assump}[Uniform ellipticity]\label{assump_elliptic}
The operator $F:\mathcal{S}(d)\to\mathbb{R}$ is $(\lambda,\Lambda)$-elliptic, i.e, there exist $0<\lambda \leq \Lambda$ such that
\[
	\lambda |N|\leq F(M)-F(M+N) \leq \Lambda |N|,
\]
for every $M,N\in\mathcal{S}(d)$, with $N\geq 0$.
\end{assump}
A well-known consequence of \ref{assump_elliptic} is the uniform Lipschitz regularity of $F$ (see for example \cite[Chapter 2]{Caffarelli-Cabre1995}). 

Next, let $G_i(u,Du)\subset B_1$, $i=1, ..., N$ be disjoint sets and define \linebreak $G_0(u,Du):=B_1\setminus \bigcup_{i=0}^N G_i(u,Du)$. Assume the exponent $\beta$ has the form
\begin{align*}
    \beta(x,u,Du)=\sum_{i=0}^N \beta_i(x)\chi_{G_i(u,Du)}.
\end{align*}
We now make some assumptions on $\beta_i$. First, assume they have a modulus of continuity which decays at the origin as $o(\ln(1/t)^{-1})$.
\begin{assump}[Uniform continuity of the exponents]\label{assump_continuity_stronger}
The exponents $\beta_i: B_1\to \Rr$ have modulus of continuity satisfying
\begin{align*}
    \limsup_{t\to 0}\ln\left(\frac{1}{t}\right)\omega(t)=0.
\end{align*}
\end{assump}
Note that \ref{assump_continuity_stronger} is equivalent to the following statement. For every \linebreak $0<r<e^{-1}$, the following holds
\begin{align*}
    \limsup_{k\to\infty}k\omega(r^k)=0.
\end{align*}
Hence, by definition, for every $\varepsilon>0$ there exists $\delta_1>0$ such that if $\rho\leq\delta_1$, then for every $k\in \Nn$,
\begin{align*}
    k\ln\left(\frac{1}{\rho}\right)\omega(\rho^k)\leq \varepsilon.
\end{align*}
Since $\rho$ will be chosen to be small, we can assume $\rho<e^{-1}$, and it follows that
\begin{align*}
    k\omega(\rho^k)\leq \varepsilon.
\end{align*}
Now, by defining $\varepsilon=\frac{\alpha_0-\alpha}{2}$ (it will become clear later why we make this choice) we also fix $\delta_1$ such that if $\rho\leq \delta_1$, then for every $k\in \Nn$, 
\begin{align}\label{Equation_bound_continuity_modulus}
    k\omega(\rho^k)\leq \frac{\alpha_0-\alpha}{2}.
\end{align}
This $\delta_1$ depends only on the continuity modulus $\omega$, the universal exponent $\alpha_0$ introduced in Remark \ref{Remark_Krylov_Safanov}, and the exponent $\alpha$ which is defined in Theorem \ref{Theorem_pointwise_regularity}.

\begin{Remark}
To emphasize this idea, let's consider some concrete examples. Suppose $\omega(t)=t^{1/2}$ and choose $\alpha$ such that $\varepsilon=\frac{\alpha_0-\alpha}{2}=1/100$. Then one can calculate that \eqref{Equation_bound_continuity_modulus} holds for $\rho\leq 4.7\times 10^{-7}$. If $\varepsilon=1/1000$, then we need $\rho\leq 2.55\times 10^{-9}$. These are numbers that depend only on these quantities and can be calculated, provided we know the expression of $\omega$ explicitly.
\end{Remark}

An example of a modulus of continuity satisfying \ref{assump_continuity_stronger} is
\begin{align*}
    \omega(t)=\ln\left(\frac{1}{t}\right)^{-p},
\end{align*}
with $p>1$.

Finally, we assume that the exponents are bounded uniformly from above and below.
\begin{assump}[Boundedness of the exponents]\label{assump_boundedness}
There exist constants $\beta_m$ and $\beta_M$ such that
\begin{align*}
    0\leq\beta_m\leq\beta_i(\cdot)\leq \beta_M<1.
\end{align*}
\end{assump}

In the following proposition, we improve slightly the usual proof of H\"older continuity of the gradient to the case with variable exponents.

\begin{Proposition}\label{Proposition_Holder_characterization}
Suppose we can find $r<1$ and sequences of affine functions $\ell_k(x)=a_k+b_k\cdot x $ and exponents $\alpha_k\uparrow \alpha$, such that $(\alpha_k-\alpha)=o(k)$ and
\begin{align}\label{Equation_holder_proposition}
    \left\|u-\ell_k\right\|_{L^\infty(B_{r^k}(x_0))}\leq K r^{k(1+\alpha_k)}.
\end{align}
Then $u\in C^{1,\alpha}(x_0)$ with constant $C(r)K$ and $0<\alpha<1$.
\end{Proposition}

\begin{proof}
Assume without loss of generality that $x_0=0$. The idea is that $\ell_k\to \ell$ uniformly, where $\ell$ satisfies the desired characterization.

Consider the first order scaling $f_r(x):=\frac{1}{r}f(rx)$. We have, by assumption,
\begin{align*}
    \left\|\ell_{k+1}-\ell_k\right\|_{L^\infty(B_{r^{k+1}})}&\leq \left\|u-\ell_{k+1}\right\|_{L^\infty(B_{r^{k+1}})}+\left\|u-\ell_k\right\|_{L^\infty(B_{r^{k+1}})}\\
    &\leq Kr^{(k+1)(1+\alpha_{k+1})}+Kr^{k(1+\alpha_{k})}\\
    &\leq Kr^{k(1+\alpha_k)}\left(r^{k(\alpha_{k+1}-\alpha_k)}+1\right)\\
    &\leq 2Kr^{k(1+\alpha_k)},
\end{align*}
since $\alpha_{k+1}-\alpha_k\geq 0$ and $r<1$.
First order scaling gives
\begin{align*}
    \left\|\left(\ell_{k+1}\right)_{r^k}-\left(\ell_{k}\right)_{r^k}\right\|_{L^\infty(B_1)}\leq 2Kr^{k\alpha_k}.
\end{align*}
Clearly, this estimate in $B_1$ implies the following estimates on the coefficients (up to a different $K$)
\begin{align}\label{boundcoefficients}
    \left|a_{k+1}-a_k\right|&\leq Kr^{k(1+\alpha_k)},\nonumber\\
    \left|b_{k+1}-b_k\right|&\leq Kr^{k\,\alpha_k}.
\end{align}
Since these are Cauchy sequences, we have that
\begin{align*}
    a_k\to a, \; b_k\to b,
\end{align*}
respectively in $\mathbb{R}$ and $\mathbb{R}^d$. It now follows that $\ell_k\to \ell$ in $L^\infty(B_1)$, where $\ell=a+b\cdot x$. Using these estimates, we get
\begin{align*}
    \left\|u-\ell\right\|_{L^\infty(B_{r^k})}&\leq \left\|u-\ell_k\right\|_{L^\infty(B_{r^k})}+|a_k-a|+r^k|b_k-b|\\
    &\leq CKr^{k(1+\alpha_k)},
\end{align*}
where $C$ depends only on $r$. Now, we apply the usual discretization strategy: for an arbitrary $0<R<1$, there exists $k\in\mathbb{N}$ such that $r^{k+1}\leq R<r^k$. Then,
\begin{align*}
    \left\|u-\ell\right\|_{L^\infty(B_R)}&\leq  \left\|u-\ell\right\|_{L^\infty(B_{r^k})}\leq CKr^{k(1+\alpha_k)} \\
    &\leq r^{k(\alpha_k-\alpha)}CKr^{k(1+\alpha)}.
\end{align*}
Now we use the fast convergence $\alpha_k\to \alpha$, so that 
\begin{align*}
    \lim_{k\to \infty}|k(\alpha_k-\alpha)|=0.
\end{align*}
Then $r^{k(\alpha_k-\alpha)}<C_1$ and therefore we get the desired inequality
\begin{align*}
    \left\|u-\ell\right\|_{L^\infty(B_R)}&\leq C(r)KR^{1+\alpha}.
\end{align*}
\end{proof}

To conclude this introductory section, we present a simple proof of Theorem   \ref{Theorem_local_regularity}, using the results from \cite{Huaroto-Pimentel-Rampasso-Swiech2020}.

\begin{Lemma}\label{Lemma_general_exponent}
Let $u\in C(B_1)$ be a viscosity solution to the equation
\begin{align}\label{Equation_general_exponent_main_equation}
    |Du|^{\beta(x,u,Du)}F\left(D^2 u\right)=f(x),
\end{align}
with $\beta(x,t,p)\in[\beta_m,\beta_M]$ and $0\leq\beta_m\leq \beta_M$.

Then $u$ is a viscosity subsolution to the equation
\begin{align}\label{Equation_general_exponent_subsolution}
    \min\Big\{  &\left|Du\right|^{\beta_m} F\left(D^2u\right),\left|Du\right|^{\beta_M} F\left(D^2u\right)\Big\}\nonumber \\
    & \leq\left\|f\right\|_{L^\infty{(B_1)}},
\end{align}
and a viscosity supersolution to the equation
\begin{align}\label{Equation_general_exponent_supersolution}
    \max\Big\{  &\left|Du\right|^{\beta_m} F\left(D^2u\right),\left|Du\right|^{\beta_M} F\left(D^2u\right)\Big\}\nonumber \\
    &\geq -\left\|f\right\|_{L^\infty{(B_1)}}.
\end{align}
\end{Lemma}
\begin{proof}
We prove only that if $u$ is a viscosity subsolution to \eqref{Equation_general_exponent_main_equation}, then it is a subsolution to \eqref{Equation_general_exponent_subsolution}, noting that the remaining case follows similarly.

Let $\varphi\in C^2(B_1)$ be such that $u-\varphi$ has a local maximum at $x_0$. Then
\begin{align*}
    |D\varphi(x_0)|^{\beta(x_0,u(x_0),D\varphi(x_0))}F\left(D^2 \varphi(x_0)\right)\leq f(x_0).
\end{align*}
Thus, depending on whether $|D\varphi(x_0)|\geq 1$ or $|D\varphi(x_0)|<1$ one of the following must hold, respectively
\begin{align*}
    &|D\varphi(x_0)|^{\beta_m}F\left(D^2 \varphi(x_0)\right)\leq \left\|f\right\|_{L^\infty{(B_1)}},\\
    &|D\varphi(x_0)|^{\beta_M}F\left(D^2 \varphi(x_0)\right)\leq \left\|f\right\|_{L^\infty{(B_1)}},
\end{align*}
provided $F\left(D^2 \varphi(x_0)\right)\geq 0$ (clearly if this is not the case, both inequalities are trivially verified). In either case, we have
\begin{align*}
    \min\Big\{  &|D\varphi(x_0)|^{\beta_m}F\left(D^2 \varphi(x_0)\right),|D\varphi(x_0)|^{\beta_M}F\left(D^2 \varphi(x_0)\right)\Big\} \\
    & \leq\left\|f\right\|_{L^\infty{(B_1)}}.
\end{align*}
Hence, we have proved that $u$ is a subsolution of \eqref{Equation_general_exponent_subsolution}.

\end{proof}
This simple result places the equation \eqref{Equation_general_exponent_main_equation} in the framework of \cite{Huaroto-Pimentel-Rampasso-Swiech2020} with $\theta_1=\beta_m$ and $\theta_2=\beta_M$ (see Proposition 1 therein). A direct application of \cite[Theorem 2]{Huaroto-Pimentel-Rampasso-Swiech2020} yields local regularity $u\in C^{1,\alpha}(B_{1/2})$ with 
\begin{align*}
    \alpha=\min\left\{ \alpha_0^-, \frac{1}{1+\beta_M} \right\},
\end{align*}
together with the estimate
\begin{align*}
    [u]_{C^{1,\alpha}(B_{1/2})}\leq C\left(\left\|u\right\|_{L^\infty(B_1)}+\left\|f\right\|_{L^\infty(B_1)}\right),
\end{align*}
which implies Theorem \ref{Theorem_local_regularity}.

The remaining of this paper is devoted to proving Theorem \ref{Theorem_pointwise_regularity}. In the next section, we begin our analysis.

\section{Scaling properties}

The following result disconnects the dependence of the exponents on the solution, by separating the possible cases.

\begin{Lemma}\label{Lemma_sub_super_solution}
Let $u\in C(B_1)$ be a viscosity solution to the perturbed equation
\begin{align}\label{Equation_simplifying_lemma_main_equation}
    |Du+p|^{\beta(x,u,Du)}F\left(D^2 u\right)=f(x),
\end{align}
with $\beta$ given by \eqref{Equation_exponent_expression}. Assume that assumptions \ref{assump_elliptic}, \ref{assump_continuity_stronger} and \ref{assump_boundedness} are in force. 

Then $u$ is a viscosity subsolution to the equation
\begin{align}\label{Equation_simplifying_lemma_subsolution}
    \min_{i=0, ..., N}\Big\{  &\left|Du+p\right|^{\beta_i(x)} F\left(D^2u\right)\Big\}\leq\left\|f\right\|_{L^\infty{(B_1)}},
\end{align}
and a viscosity supersolution to the equation
\begin{align}\label{Equation_simplifying_lemma_supersolution}
    \max_{i=0, ..., N}\Big\{  &\left|Du+p\right|^{\beta_i(x)} F\left(D^2u\right)\Big\}\geq -\left\|f\right\|_{L^\infty{(B_1)}}.
\end{align}
\end{Lemma}
\begin{proof}
We prove only that if $u$ is a viscosity subsolution to \eqref{Equation_simplifying_lemma_main_equation}, then it is a subsolution to \eqref{Equation_simplifying_lemma_subsolution}, noting that the remaining case follows similarly.

Let $\varphi\in C^2(B_1)$ be such that $u-\varphi$ has a local maximum at $x_0$. Then
\begin{align*}
    |D\varphi(x_0)+p|^{\beta(x_0,u,Du)}F\left(D^2 \varphi(x_0)\right)\leq f(x_0),
\end{align*}
where
\begin{align*}
    \beta(x,u,Du)=\sum_{i=0}^N \beta_i(x)\chi_{G_i(u,Du)}.
\end{align*}
We recall that by Theorem \ref{Theorem_local_regularity} we know that $u\in C^{1,\alpha}$, hence $G_i(u,Du)$ are well-defined. Since $G_i, \, i=0, ..., N$, form a disjoint partition of $B_1$, there is a unique $i_0\in\{0, ..., N\}$ such that $x_0\in G_{i_0}(u,Du)$. Thus
\begin{align*}
    &|D\varphi(x_0)+p|^{\beta_{i_0}(x_0)}F\left(D^2 \varphi(x_0)\right)\leq f(x_0).
\end{align*}
In particular, we have
\begin{align*}
    \min_{i=0, ..., N}\Big\{  |D\varphi(x_0)+p|^{\beta_i(x_0)}F\left(D^2 \varphi(x_0)\right)\Big\}\leq\left\|f\right\|_{L^\infty{(B_1)}}.
\end{align*}
Hence, we have proved that $u$ is a subsolution of \eqref{Equation_simplifying_lemma_subsolution}.

\end{proof}

The following result states that to prove Theorem \ref{Theorem_pointwise_regularity}, we can assume a smallness regime, without loss of generality. It provides a tangential path between our equation and the homogeneous one.

\begin{Proposition}[Smallness regime]
Let $u$ be a subsolution to the equation
\begin{align}\label{Equation_translated_subsolution}
    \min_{i=0, ..., N}\Big\{  \left|Du\right|^{\beta_i(x)} F\left(D^2u\right) \Big\}\leq\left\|f\right\|_{L^\infty(B_1)},
\end{align}
and a supersolution to the equation
\begin{align}\label{Equation_translated_supersolution}
    \max_{i=0, ..., N}\Big\{  \left|Du\right|^{\beta_i(x)} F\left(D^2u\right) \Big\}\geq-\left\|f\right\|_{L^\infty(B_1)}.
\end{align}
satisfying
\begin{align*}
    [u]_{C^{1,\alpha}(x_0)}\leq C,
\end{align*}
under the assumption that $\left\|u\right\|_{L^\infty(B_1)}\leq 1$ and $\left\|f\right\|_{L^\infty(B_1)}\leq \varepsilon_0$ , where $C$ and $\varepsilon_0$ are universal constants. Then, Theorem \ref{Theorem_pointwise_regularity} holds.
\end{Proposition}
\begin{proof}
Let $\overline{u}(x)=Ku(x)$ where
\begin{align*}
    K:=\left(\left\|u\right\|_{L^\infty(B_1)}+\frac{\left\|f\right\|_{L^\infty(B_1)}}{\varepsilon_0}\right)^{-1}.
\end{align*}
Note that we can assume $K\leq 1$, since otherwise we are already in the smallness regime and we can just take $K=1$.

The function $\overline{u}$ is a viscosity subsolution to
\begin{align*}
    \min_{i=0, ..., N}\Big\{& K^{-\beta_i(x)}|D\overline{u}|^{\beta_i(x)}KF\left(K^{-1} D^2\overline{u}\right)\Big\}\leq K \left\|f\right\|_{L^\infty(B_1)},
\end{align*}
which implies
\begin{align*}
    \min_{i=0, ..., N}\Big\{& |D\overline{u}|^{\beta_i(x)}\overline{F}( D^2\overline{u})\Big\}\leq \max_{i=0, ..., N}\left\{K^{1+\beta_i(x)}\right\} \left\|f\right\|_{L^\infty(B_1)},
\end{align*}
where $\overline{F}( M):=KF\left(K^{-1}M\right)$ still satisfies \ref{assump_elliptic}. Since $K\leq 1$, we immediately get
\begin{align*}
    \min_{i=0, ..., N}\Big\{& |D\overline{u}|^{\beta_i(x)}\overline{F}( D^2\overline{u})\Big\}\leq\varepsilon_0.
\end{align*}
Similarly, we get
\begin{align*}
    \max_{i=0, ..., N}\Big\{& |D\overline{u}|^{\beta_i(x)}\overline{F}( D^2\overline{u})\Big\}\geq-\varepsilon_0.
\end{align*}
Since $\left\|\overline{u}\right\|_{L^\infty(B_1)}\leq 1$, we note that the smallness assumptions are now satisfied. Hence, if we verify that
\begin{align*}
    [\overline{u}]_{C^{1,\alpha}(B_1)}\leq C,
\end{align*}
we can infer that
\begin{align}\label{Equation_smallness_assumption_norm}
    [u]_{C^{1,\alpha}(B_1)}\leq C_1\left(\left\|u\right\|_{L^\infty(B_1)}+\left\|f\right\|_{L^\infty(B_1)}\right),
\end{align}
where $C_1$ depends on $\varepsilon_0$, which will be fixed universally.

\end{proof}
\begin{Remark}
The choice of $K$ in the previous proof differs from the literature (see for example  \cite{Huaroto-Pimentel-Rampasso-Swiech2020}). The observation that $K\leq 1$ allows  us to obtain the simple estimate \eqref{Equation_smallness_assumption_norm}.
\end{Remark}

In the following section we obtain improved regularity.

\section{H\"older continuity}
In this section, we obtain a compactness result for solutions. This result is essential when studying stability since it will allow us to obtain convergence of sequences of solutions.

We start by stating the maximum principle for viscosity solutions, Theorem 3.2 of \cite{Crandall-Ishii-Lions1992}.
\begin{Proposition}[Maximum principle]\label{Proposition_Maximum_Principle_CIL}
Let $\Omega$ be a bounded domain and $G, \, H\in C\left(B_1\times \Rr^d\times \mathcal{S}(d)\right)$ be degenerate elliptic. Let $u_1$ be a viscosity subsolution of $G\left(x, Du_1, D^2u_1\right)=0$ and $u_2$ be a viscosity supersolution of $H\left(x, Du_2, D^2u_2\right)=0$ in $\Omega$. Let $\varphi\in C^2(\Omega\times \Omega)$. Define $v:\Omega\times \Omega\to \Rr$ by
\begin{align*}
    v\left(x,y\right):=u_1(x)-u_2(y).
\end{align*}
Suppose further that $\left(\overline{x}, \overline{y}\right)\in\Omega\times \Omega$. Then, for every $\iota>0$, there exist matrices $X$ and $Y$ in $\mathcal{S}\left(d\right)$ such that
\begin{align*}
    G\left(\overline{x}, D_x\varphi\left(\overline{x}, \overline{y}\right), X\right)\leq 0 \leq H\left(\overline{y}, -D_y\varphi\left(\overline{x}, \overline{y}\right), Y\right),
\end{align*}
and the matrix inequality 
\begin{align*}
    -\left(\frac{1}{\iota}+\left\|A\right\|\right)I\leq 
    \begin{pmatrix}
  X & 0 \\
  0 & -Y 
 \end{pmatrix}
 \leq A+\iota A^2
\end{align*}
holds true, where $A:=D^2\varphi\left(\overline{x}, \overline{y}\right)$.
\end{Proposition}

We proceed by stating a result from \cite{Imbert-Silvestre2016}, which we present in the following simplified form.
\begin{Proposition}
Let $u\in C(B_1)$ be a bounded viscosity subsolution to equation
\begin{align*}
    \mathcal{P}^-_{\lambda, \Lambda}\left(D^2u\right)-|Du|=0, \mbox{ in } \{|Du|>\gamma\}
\end{align*}
and a viscosity supersolution to equation
\begin{align*}
    \mathcal{P}^+_{\lambda, \Lambda}\left(D^2u\right)+|Du|=0, \mbox{ in } \{|Du|>\gamma\}.
\end{align*}

Then $u\in C^{\theta}_{loc}(B_1)$ and, for every $0<\tau<1$, there exists $C>0$ such that 
\begin{align*}
    \left\|u\right\|_{C^{\theta}(B_\tau)}\leq C.
\end{align*}
The constant $\theta$ depends only on $d, \lambda, \Lambda$ and $C$ depends only on $d, \lambda, \Lambda$, $ \gamma, \left\|u\right\|_{L^\infty(B_1)}, \tau$.
\end{Proposition}

Intuitively, in the set where the gradient  of a function $u$ is bounded, the function is already Lipschitz. The idea behind the previous result is that if $u$ is a solution of an elliptic equation in the set where its gradient is very large, then we are able to obtain improved regularity.

This proposition will imply H\"older continuity of solutions to \eqref{Equation_simplifying_lemma_subsolution} and \eqref{Equation_simplifying_lemma_supersolution} in the case where $|p|$ is sufficiently small. More precisely, let $A_0>1$ (to be fixed) be such that $|p|<A_0$. We claim that $u$ is a viscosity subsolution to
\begin{align}\label{Equation_IS_subsolution}
    F\left(D^2u\right)-|Du|=0, \mbox{ in } \{|Du|>2A_0\}.
\end{align}
Indeed, take $\varphi\in C^2$ such that $u-\varphi$ has a local maximum at \linebreak $x_0\in \{|Du|>2A_0\}$. Then $|D\varphi(x_0)|>2A_0$ and therefore \linebreak $|D\varphi(x_0)+p|\geq A_0>1$. From \eqref{Equation_simplifying_lemma_subsolution} we have
\begin{align*}
    \min_{i=0, ..., N}\Big\{&|D\varphi(x_0)+p|^{\beta_i(x_0)} F\left(D^2\varphi(x_0\right))\Big\}\leq \left\|f\right\|_{L^\infty(B_1)},
\end{align*}
which implies 
\begin{align*}
    F\left(D^2\varphi(x_0\right))\leq \left\|f\right\|_{L^\infty(B_1)}\leq |D\varphi(x_0)|
\end{align*}
since we are under the assumption $\left\|f\right\|_{L^\infty(B_1)}\leq \varepsilon_0$ and $\varepsilon_0$ will be chosen very small. Hence, from uniform ellipticity and recalling that $F(0)=0$,
\begin{align*}
    \mathcal{P}^-_{\lambda, \Lambda}\left(D^2\varphi(x_0\right))-|D\varphi(x_0)|\leq F\left(D^2\varphi(x_0\right))-|D\varphi(x_0)|\leq 0.
\end{align*}
We verified that $u$ is a viscosity subsolution to \eqref{Equation_IS_subsolution}. In a similar way, we prove that $u$ is a viscosity supersolution to 
\begin{align*}
    \mathcal{P}^+_{\lambda, \Lambda}\left(D^2\varphi(x_0\right))+|Du|=0,
\end{align*}
in $\{|Du|> 2A_0\}$. Hence, we proved the following corollary.
\begin{Corollary}\label{Corollary_Holder_continuity_small_p}
Let $u\in C(B_1)$ be a bounded viscosity subsolution to \eqref{Equation_simplifying_lemma_subsolution} and a supersolution to \eqref{Equation_simplifying_lemma_supersolution}.  Assume \ref{assump_elliptic}, \ref{assump_continuity_stronger} and \ref{assump_boundedness} are in force, let $\left\|u\right\|_{L^\infty(B_1)}\leq 1$, $\left\|f\right\|_{L^\infty(B_1)}\leq \varepsilon_0$ and assume further that $\left|p\right|<A_0$. The constants $\varepsilon_0$ and $A_0$ will be fixed in the sequel.

Then $u\in C^{\theta}_{loc}(B_1)$ for some $\theta\in (0,1)$, depending only on $d, \lambda, \Lambda$. In addition, for every $0<\tau<1$, there exists $C>0$ such that
\begin{align*}
    \left\|u\right\|_{C^\theta(B_\tau)}\leq C,
\end{align*}
where $C=C(d, \lambda, \Lambda, A_0,\tau)$.
\end{Corollary}

In the following lemma we obtain H\"older continuity for arbitrary $p\in\Rr^d$, which concludes this section.

\begin{Lemma}[$C^\theta$ regularity]\label{Lemma_Holder_continuity_arbitrary_p}
Let $u\in C(B_1)$ be a bounded viscosity subsolution to \eqref{Equation_simplifying_lemma_subsolution} and a supersolution to \eqref{Equation_simplifying_lemma_supersolution}.  Assume \ref{assump_elliptic}, \ref{assump_continuity_stronger} and \ref{assump_boundedness} are in force and let $\left\|u\right\|_{L^\infty(B_1)}\leq 1$ and $\left\|f\right\|_{L^\infty(B_1)}\leq \varepsilon_0$, to be fixed universally. 

Then $u\in C^{\theta}_{loc}(B_1)$ for some $\theta\in (0,1)$, depending only on $d, \lambda, \Lambda$. In addition, for every $0<\tau<1$, there exists $C>0$ such that
\begin{align*}
    \left\|u\right\|_{C^\theta(B_\tau)}\leq C,
\end{align*}
where $C=C(d, \lambda, \Lambda)$.
\end{Lemma}

\begin{proof}
We begin by using Proposition \ref{Proposition_Maximum_Principle_CIL} to obtain a subjet and a superjet satisfying the estimate \eqref{Equation_Pucci_estimate_Holder_lemma} below. This was done in \cite[Proposition 7]{Huaroto-Pimentel-Rampasso-Swiech2020} but for completion we replicate the proof.

Fix $0<r<\frac{1-\tau}{2}$ and define
\begin{align*}
    \omega(t)=t-\frac{t^2}{2}.
\end{align*}
For constants $L_1$, $L_2>0$ and $x_0\in B_r$, we set
\begin{align*}
    L:=\sup_{x,y\in B_r(x_0)}\left[ u(x)-u(y)-L_1\omega\left(\left|x-y\right|\right)-L_2\left(\left|x-x_0\right|^2+\left|y-x_0\right|^2\right)  \right]
\end{align*}

Set $A_0=4L_1$ and assume $|p|\geq A_0$.

We aim at establishing that there exist constants $L_1$ and $L_2$, independent of $x_0$, for which $L\leq 0$. This immediately implies that $u$ is Lipschitz continuous in $B_\tau$ by taking $x_0=x$.

We argue by contradiction. Suppose there exists $x_0\in B_\tau$ for which $L>0$, regardless of the choices of $L_1$ and $L_2$. Consider the auxiliary functions $\psi, \phi: \overline{B}_1\times \overline{B}_1\to \Rr$ given by
\begin{align*}
    \psi\left(x,y\right):=L_1\omega\left(\left|x-y\right|\right)+L_2\left(\left|x-x_0\right|^2+\left|y-x_0\right|^2\right)
\end{align*}
and
\begin{align*}
    \phi\left(x,y\right):=u(x)-u(y)-\psi\left(x,y\right).
\end{align*}
Let $\left(\overline{x},\overline{y}\right)$ be a point where $\phi$ attains its maximum. Then
\begin{align*}
    \phi\left(\overline{x}, \overline{y}\right)=L>0
\end{align*}
and
\begin{align*}
    L_1\omega\left(\left|\overline{x}-\overline{y}\right|\right)+L_2\left(\left|\overline{x}-x_0\right|^2+\left|\overline{y}-x_0\right|^2\right)\leq 2.
\end{align*}
Set
\begin{align*}
    L_2:=\left(\frac{4\sqrt{2}}{r}\right)^2.
\end{align*}
Then
\begin{align*}
    \left|\overline{x}-x_0\right|+\left|\overline{y}-x_0\right|\leq \frac{r}{2},
\end{align*}
which implies that $\overline{x}, \overline{y}\in B_r(x_0)$. In addition, $\overline{x}\neq \overline{y}$, since if this isn't the case we would conclude that $L\leq 0$.

We now use Proposition \ref{Proposition_Maximum_Principle_CIL} to ensure the existence of a subjet $(\xi_x, X)$ of $u$ at $\overline{x}$ and a superjet $(\xi_y, Y)$ of $u$ at $\overline{y}$ with
\begin{align*}
    &\xi_x:=D_x\psi\left(\overline{x}, \overline{y}\right)=L_1\omega'\left(\left|\overline{x}-\overline{y}\right|\right)\sigma+2L_2\left(\overline{x}-x_0\right),\\
    &\xi_y:=-D_y\psi\left(\overline{x}, \overline{y}\right)=L_1\omega'\left(\left|\overline{x}-\overline{y}\right|\right)\sigma-2L_2\left(\overline{x}-x_0\right),
\end{align*}
where 
\begin{align*}
    \sigma:=\frac{\overline{x}-\overline{y}}{\left|\overline{x}-\overline{y}\right|}.
\end{align*}
Since $\omega'\left(\left|\overline{x}-\overline{y}\right|\right)\leq 1$,
\begin{align}\label{Equation_xi_bound_1}
    \left|\xi_x\right|\leq L_1+\frac{L_2}{2}\leq 2L_1,
\end{align}
and
\begin{align}\label{Equation_xi_bound_2}
    \left|\xi_y\right|\leq L_1+\frac{L_2}{2}\leq 2L_1
\end{align}
for $L_1$ large enough.

In addition, the matrices $X$ and $Y$ satisfy the inequality
\begin{align}\label{Equation_matrix_inequality_holder_lemma}
    \begin{pmatrix}
    X & 0 \\
    0 & -Y
    \end{pmatrix}\leq 
    \begin{pmatrix}
    Z & -Z \\
    -Z & Z
    \end{pmatrix} 
    +\left(2L_2+\iota\right)I,
\end{align}
for
\begin{align*}
    Z:=L_1\omega''\left(\left|\overline{x}-\overline{y}\right|\right)\sigma \otimes\sigma+ L_1\frac{\omega'\left(\left|\overline{x}-\overline{y}\right|\right)}{\left|\overline{x}-\overline{y}\right|}\left(I-\sigma\otimes\sigma\right),
\end{align*}
where $0<\iota\ll 1$ depends solely on the norm of $Z$.

Next we apply the matrix inequality  \eqref{Equation_matrix_inequality_holder_lemma} to special vectors as to obtain information about the eigenvalues of $X-Y$. First, apply it to vectors of the form $\left(z,z\right)\in \Rr^{2d}$ to get
\begin{align*}
    z\cdot \left(X-Y\right)z\leq \left(4L_2+2\iota\right)\left|z\right|^2
\end{align*}
which implies that all eigenvalues of $X-Y$ are less than or equal to $4L_2+2\iota$.

Now we apply \eqref{Equation_matrix_inequality_holder_lemma} to the vector $\overline{z}=\left(\sigma,-\sigma\right)$ to obtain
\begin{align*}
    \sigma\cdot \left(X-Y\right)\sigma\leq& 4L_2+2\iota+4L_1\omega''\left(\left|\overline{x}-\overline{y}\right|\right)\\
    =&4L_2+2\iota-4L_1.
\end{align*}
We thus conclude that at least one eigenvalue of $X-Y$ is below $4L_2+2\iota-4L_1$, which will be a negative number, provided we choose $L_1$ large enough.

Evaluating the minimal Pucci operator on $X-Y$, we get
\begin{align}\label{Equation_Pucci_estimate_Holder_lemma}
    \mathcal{P}^-_{\lambda,\Lambda}(X-Y)\geq& 4\lambda L_1-\left(\lambda\left(d-1\right)\Lambda\right)\left(4L_2+2\iota\right)\nonumber\\
    \geq & 3\lambda L_1
\end{align}
for $L_1$ even larger, if necessary. Furthermore, these jets satisfy the viscosity estimates
\begin{align}\label{Equation_subjet_estimate_Holder_continuity}
    \min_{i=0, ..., N}\Big\{&|p+\xi_x|^{\beta_i(\overline{x})} F(X)\Big\}\leq \varepsilon_0,
\end{align}
and
\begin{align}\label{Equation_superjet_estimate_Holder_continuity}
    \max_{i=0, ..., N}\Big\{&|p+\xi_y|^{\beta_i(\overline{x})} F(Y)\Big\}\geq -\varepsilon_0.
\end{align}
Since we fixed $A_0=4L_1$ and assumed $\left|p\right|\geq A_0$, this together with \eqref{Equation_xi_bound_1} and \eqref{Equation_xi_bound_2} imply
\begin{align*}
    &\left|p+\xi_x\right|\geq 2L_1>1,\\
    &\left|p+\xi_y\right|\geq 2L_1>1.
\end{align*}
Hence, \eqref{Equation_subjet_estimate_Holder_continuity} and \eqref{Equation_superjet_estimate_Holder_continuity} imply, respectively, 
\begin{align*}
    F(X)\leq \varepsilon_0
\end{align*}
and
\begin{align*}
    F(Y)\geq -\varepsilon_0.
\end{align*}
Combining these inequalities with \eqref{Equation_Pucci_estimate_Holder_lemma} by means of uniform ellipticity, we get
\begin{align*}
    3\lambda L_1\leq 2\varepsilon_0,
\end{align*}
which is clearly a contradiction, provided we choose $L_1$ large enough.

This concludes the proof for the case $|p|\geq A_0$, which combined with Corollary \ref{Corollary_Holder_continuity_small_p} completes the proof.
\end{proof}

With compactness available, we proceed with a key step in our tangential analysis.

\section{Approximation Lemma}

We present an approximation lemma for the perturbed equation.
\begin{Lemma}[Approximation Lemma]\label{Lemma_approximation}
For every $0<\delta<1$, there exists $\varepsilon_0>0$ such that, if $u\in C(B_1)$ is a viscosity subsolution to \eqref{Equation_simplifying_lemma_subsolution} and a viscosity supersolution to \eqref{Equation_simplifying_lemma_supersolution} with $p=0$,  satisfying $\left\|u\right\|_{L^\infty(B_1)}\leq 1$ and $\left\|f\right\|_{L^\infty(B_1)}\leq \varepsilon_0$, then one can find a function $h$ which is a viscosity solution to $\overline{F}(D^2h)=0$ for some $\overline{F}$ satisfying assumption \ref{assump_elliptic}, such that
\begin{align*}
    \left\|u-h\right\|_{L^\infty(B_{1/2})}\leq \delta.
\end{align*}
Such a function $h$ satisfies
\begin{align*}
    \left\|h\right\|_{C^{1,\alpha_0}(B_{1/2})}\leq C\left\|h\right\|_{L^\infty(B_{3/4})}.
\end{align*}
\end{Lemma}
\begin{proof}
We argue by contradiction. For simplicity, we split the proof in steps.

\medskip

\noindent{\bf Step 1 - } Assume that there exist $\delta_0>0$ and sequences $\left(u_n\right)_n$, $\left(F_n\right)_n$ and $\left(\beta_i^n\right)_n$  such that
\begin{enumerate}
    \item $\left\|u_n\right\|_\infty\leq 1$;
    \item $F_n$ satisfy \ref{assump_elliptic};
    \item $\beta_i^n$ satisfy \ref{assump_continuity_stronger} and \ref{assump_boundedness};
\end{enumerate}
linked together by the equations
\begin{align*}
    \min_{i=0, ..., N} \Big\{& \left|Du_n+p_n\right|^{\beta_i^n(x)}F_n\left(D^2 u_n\right)\Big\}\leq \frac{1}{n}
\end{align*}
and
\begin{align*}
    \max_{i=0, ..., N} \Big\{& \left|Du_n+p_n\right|^{\beta_i^n(x)}F_n\left(D^2 u_n\right)\Big\}\geq -\frac{1}{n},
\end{align*}
in the viscosity sense, in $B_1$; however, for every function $h\in C^{1,\alpha_0}$, it holds
\begin{align}\label{Equation_contradiction_un_C1alpha0}
    \left\|u_n-h\right\|_{L^\infty(B_{1/2})}>\delta_0.
\end{align}

\medskip

\noindent{\bf Step 2 - } Since $u_n$ are equibounded in $C^\theta(B_{9/10})$, by the Arzelà-Ascoli Theorem they will converge, up to a subsequence, locally uniformly to \linebreak $u_\infty\in C(B_1)$.

Since $F_n$ are $(\lambda,\Lambda)$-elliptic, they are also Lipschitz continuous. Therefore, again by the Arzelà-Ascoli Theorem they will converge locally uniformly to an $(\lambda, \Lambda)$-elliptic operator $F_\infty$.

Finally, since $\beta_i^n$ satisfy assumptions \ref{assump_continuity_stronger} and \ref{assump_boundedness} they will converge locally uniformly to  continuous functions $\beta^\infty_i$, respectively.

Our goal is to prove that the limiting function $u_\infty$  is a viscosity solution to the equation $F_\infty(D^2u_\infty)=0$. We only prove that it is a subsolution, since the proof for supersolution is analogous. We will consider two cases, depending on the limit behaviour of $(p_n)_n$.

\textbf{Step 3:} Assume that $\left(p_n\right)_n$ does not admit a convergent subsequence. Then $|p_n|\to \infty$. Let $\varphi\in C^2(B_1)$ and assume that $u_\infty-\varphi$ attains a local strict maximum at $x_0\in B_1$. By contradiction, assume that 
\begin{align}\label{Equation_contradiction_approx_lemma}
    F_\infty\left(D^2\varphi (x_0)\right)>0.
\end{align}
There exists a sequence $x_n\to x_0$ such that $u_n-\varphi$ has a local maximum at $x_n$. Notice that $D\varphi(x_n)\to D\varphi (x_0)$ and $D^2\varphi(x_n)\to D^2\varphi(x_0)$. Also, by the equation satisfied by $u_n$ in the viscosity sense, we have
\begin{align*}
    \min_{i=0, ..., N} \Big\{& \left|D\varphi(x_n)+p_n\right|^{\beta_i^n(x_n)}F_n\left(D^2\varphi(x_n)\right)\Big\}\leq \frac{1}{n}.
\end{align*}
Taking $n$ large enough, we have $|D\varphi(x_n)+p_n|> 1$ and since $\beta_i^n\geq 0$, we get
\begin{align*}
    F_n\left(D^2\varphi(x_n)\right)\leq \frac{1}{n},
\end{align*}
which is inconsistent with \eqref{Equation_contradiction_approx_lemma}, when we take the limit $n\to \infty$. Therefore,
\begin{align*}
    F_\infty(D^2\varphi(x_0))\leq 0,
\end{align*}
concluding the proof for the case $|p_n|\to \infty$.

\textbf{Step 4:} Suppose now that we can extract a subsequence $p_n\to p_\infty$. Resorting to standard stability results (see for example \cite[Remarks 6.2 and 6.3]{Crandall-Ishii-Lions1992}), we conclude that $u_\infty$ is a viscosity subsolution to 
\begin{align*}
    \min_{i=0, ..., N}\Big\{&\left|p_\infty+Du_\infty\right|^{\beta_i^\infty(x)} F_\infty(D^2u_\infty)\Big\} \leq 0,
\end{align*}
and a viscosity supersolution to
\begin{align*}
    \max_{i=0, ..., N}\Big\{&\left|p_\infty+Du_\infty\right|^{\beta_i^\infty(x)} F_\infty(D^2u_\infty)\Big\} \geq 0.
\end{align*}
We can assume without loss of generality that $p_\infty=0$, i.e., assume that $u_\infty$ is a viscosity subsolution to
\begin{align*}
    \min_{i=0, ..., N}\Big\{\left|Du_\infty\right|^{\beta_i^\infty(x)} F_\infty(D^2u_\infty)\Big\} \leq 0,
\end{align*}
and a viscosity supersolution to
\begin{align*}
    \max_{i=0, ..., N}\Big\{\left|Du_\infty\right|^{\beta_i^\infty(x)} F_\infty(D^2u_\infty)\Big\}\geq 0.
\end{align*}
We now claim that these inequalities imply that $F_\infty(D^2 u_\infty)=0$. This is proved in Lemma \ref{Lemma_homogeneous_division} below.

\textbf{Step 5:} Since $F_\infty(D^2u_\infty)=0$, by Remark \ref{Remark_Krylov_Safanov} we get that $u_\infty\in C^{1,\alpha_0}(B_{1/2})$. This, together with the uniform convergence $u_n\to u_\infty$ produces a contradiction with \eqref{Equation_contradiction_un_C1alpha0}, which completes the proof.

\end{proof}

We present a homogeneous division lemma which concludes the proof of Lemma \ref{Lemma_approximation}. We follow closely the proof of \cite[Lemma 6]{Imbert-Silvestre2013}.
\begin{Lemma}\label{Lemma_homogeneous_division}
Let $u\in C(B_1)$ be a bounded viscosity subsolution to
\begin{align}\label{Equation_homogeneous_lemma_subsolution}
    \min_{i=0, ..., N}\Big\{|Du|^{\beta_i(x)}F\left(D^2u\right)\Big\}\leq 0,
\end{align}
and a viscosity supersolution to
\begin{align}\label{Equation_homogeneous_lemma_supersolution}
    \max_{i=0, ..., N}\Big\{|Du|^{\beta_i(x)}F\left(D^2u\right)\Big\}\geq 0.
\end{align}
Then $u$ is a viscosity solution to
\begin{align*}
    F\left(D^2u\right)=0.
\end{align*}
\end{Lemma}
\begin{proof}
We prove that \eqref{Equation_homogeneous_lemma_supersolution} implies $F\left(D^2u\right)\geq 0$, noting that $F\left(D^2u\right)\leq 0$ follows similarly from \eqref{Equation_homogeneous_lemma_subsolution} in a similar way.

Let $P(x)=\frac{1}{2}(x-y)\cdot N(x-y) +b\cdot (x-y) +u(y)$ be a polynomial touching $u$ strictly from below at a point $y\in B_{3/4}$. We shall assume, without loss of generality, that $y=0$ and $u(0)=0$. Then we have the estimate
\begin{align*}
    \min_{i=0, ..., N}\Big\{|b|^{\beta_i(0)}F(N)\Big\}\leq 0.
\end{align*}
If $b\neq 0$, then the result is trivial. Hence, assume otherwise. So \linebreak $P(x)=\frac{1}{2}x\cdot Nx$. We argue by contradiction, assuming that $F(N)<0$. By ellipticity, this implies that $N$ has at least one positive eigenvalue. Let $S$ be the subspace generated by the eigenvectors corresponding to the positive eigenvalues and consider the projection $P_S$ to this subspace. We consider the following perturbed test function
\begin{align*}
    \psi(x)=P(x)+\epsilon|P_Sx|, \mbox{ for } x\in \overline{B}_r.
\end{align*}
For $\epsilon$ large enough, $u-\psi$  attains a negative minimum at $B_r-B_{\frac{9}{10}r}$ (since $S$ is not empty and $u$ is continuous). Indeed, let $m=\max_{\overline{B}_r}|u-P|$ and $\epsilon=\frac{21 m}{r}$. Then
\begin{align*}
    \min_{x\in\overline{B}_{\frac{9}{10}r}}(u(x)-P(x)-\epsilon |P_Sx|)\geq -m-\frac{21 M}{r}\frac{9}{10}r=-\frac{199}{10}m;
\end{align*}
on the other hand, for $x\in \partial B_r$,
\begin{align*}
    u(x)-P(x)-\epsilon |P_Sx|\leq m-\frac{21 m}{r}r=-20m,
\end{align*}
which concludes that the minimum is negative and attained at \linebreak $x_0\in B_r\entre B_{\frac{9}{10}r}$.

Using this $\epsilon$, we claim that $P_Sx_0\neq 0$. In fact, since \linebreak $(u-\psi)(x_0)\leq (u-\psi)(x)$,
\begin{align*}
    (u-P)(x_0)-\epsilon|P_Sx_0|\leq (u-\psi)(x)
\end{align*}
if $P_Sx_0=0$, we take $x=0$ and get
\begin{align*}
    (u-P)(x_0)\leq (u-\psi)(0)=(u-P)(0)=0.
\end{align*}
But $P$ touches $u$ strictly  from below at $y=0$ which implies $(u-P)(x_0)\geq 0$ with equality only if $x_0=0$, therefore $(u-\psi)(x_0)=0$, contradicting the fact that the minimum is negative and that $x_0\neq 0$.

We proved that $|P_Sx_0|\neq 0$ which implies that $\psi$ is smooth in a neighbourhood of $x_0$. Hence, for an appropriate translation of $\psi$, call it $\Tilde{\psi}$, $u-\tilde{\psi}$ has a local minimum in $B_r$ at $x_0$. Let $B$ be the Hessian of $|P_Sx|$ at $x=x_0$. Note that since $|P_Sx|$ is a convex function, $B\geq 0$. We also have the viscosity inequality
\begin{align*}
    \max_{i=0, ..., N}\Big\{  &|Nx_0+\epsilon e_0|^{\beta_i(x_0)}F(N+\epsilon B) \Big\}\geq 0,
\end{align*}
for $e_0=P_Sx_0/|P_Sx_0|$. Note $(Nx_0+\epsilon e_0)\cdot P_Sx_0>0$, since $P_S$ is the projection into the subspace generated by the eigenvalues of $N$ associated with its positive eigenvalues. Then, by ellipticity we obtain
\begin{align*}
    F(N)\geq F(N+\epsilon B)\geq 0,
\end{align*}
which is a contradiction. Hence $F\left(D^2u\right)\geq 0$ which concludes the proof.
\end{proof}

In the next and final section we provide an iterative scheme to control the oscillation of the gradient.

\section{H\"older continuity of the gradient}

In the following lemmas, we proceed with the geometric iteration argument in a sequence of concentric, shrinking balls. The first geometric iteration follows immediately from the approximation lemma.

\begin{Lemma}\label{Lemma_first_iteration}
Let $u\in C(B_1)$ be a viscosity subsolution to \eqref{Equation_simplifying_lemma_subsolution} and a viscosity supersolution to \eqref{Equation_simplifying_lemma_supersolution} with $p=0$. Under the assumptions of Lemma \ref{Lemma_approximation}, for every $\theta<\alpha_0$, there exists a polynomial $\ell(x)=a+b\cdot x$ and a constant $0<\rho<1$, depending on $\omega$ and universal constants, such that
\begin{align*}
    \left\|u-\ell\right\|_{L^\infty(B_\rho)}\leq \rho^{1+\theta}.
\end{align*}
Furthermore, there exists a universal $C>0$ such that  $\left|a\right|\leq C$ and $\left|b\right|\leq C$.
\end{Lemma}

\begin{proof}
By Lemma \ref{Lemma_approximation}, there exists $h\in C^{1,\alpha_0}(B_{1/2})$ such that
\begin{align*}
    \left\|u-h\right\|_{L^\infty(B_{1/2})}\leq \delta,
\end{align*}
with the uniform estimate
\begin{align*}
    \left\|h\right\|_{C^{1,\alpha_0}(B_{1/2})}\leq C\left\|h\right\|_{L^\infty(B_{3/4})}.
\end{align*}
This implies that, for every $0<r\ll 1$ and for the polynomial \linebreak $\ell(x)=h(x_0)+Dh(x_0)\cdot (x-x_0)$,
\begin{align*}
    \left\|h-\ell\right\|_{L^\infty(B_r(x_0))}\leq C r^{1+\alpha_0},
\end{align*}
with $|h(x_0)|\leq C$ and $|Dh(x_0)|\leq C$ where $C$ is universal.

Let $\theta<\alpha_0$ be arbitrary and take $r=\rho$ given by
\begin{align*}
    \rho:=\min\left\{ \delta_1,\,\left(2C\right)^\frac{1}{\theta-\alpha_0} \right\},
\end{align*}
where $\delta_1$ is given implicitly in \eqref{Equation_bound_continuity_modulus}.
Finally fix $\delta=\frac{\rho^{1+\theta}}{2}$, which also fixes $\varepsilon_0$ via Lemma \ref{Lemma_approximation}. Then
\begin{align*}
    \left\|u-\ell\right\|_{L^\infty(B_\rho(x_0))}\leq& \left\|u-h\right\|_{L^\infty(B_\rho(x_0))}+\left\|h-\ell\right\|_{L^\infty(B_\rho(x_0))}\\
    \leq &\delta + C \rho^{1+\alpha_0}\leq \rho^{1+\theta}.
\end{align*}
\end{proof}

Now we iterate the previous result concentric, shrinking balls.

\begin{Lemma}[Geometric iterations]\label{Lemma_geometric_iterations_pointwise} There exist a non-decreasing sequence $(\alpha_k)_k$ and universal constants $\varepsilon_0>0$ and $\rho>0$ such that if $u$ is a viscosity subsolution of \eqref{Equation_simplifying_lemma_subsolution} and a supersolution of \eqref{Equation_simplifying_lemma_supersolution} with $p=0$, satisfying $\left\|u\right\|_{L^\infty(B_1)}\leq 1$ and $\left\|f\right\|_{L^\infty(B_1)}\leq \varepsilon_0$, there exist polynomials $\ell_k(x)=a_k+b_k\cdot x$ such that
\begin{align}\label{Estimate_approximation_geometric_iteration_freeboundary2}
    \left\|u-\ell_k\right\|_{L^\infty(B_{\rho^k}(x_0))}\leq \rho^{k(1+\alpha_k)},
\end{align}
and
\begin{align}\label{Estimate_coefficients_geometric_iteration_freeboundary2}
    |a_k-a_{k-1}|+\rho^{k-1}|b_k-b_{k-1}|\leq C_e\rho^{(k-1)(1+\alpha_{k-1})}.
\end{align}
Furthermore, the sequence $\left(\alpha_k\right)_k$ converges to
\begin{align*}
    \alpha:=\min_{i=0, ..., N}\left\{\alpha_0^-,\frac{1}{1+\beta_i(x_0)}\right\}
\end{align*}
and
\begin{align}\label{Equation_rate_convergence_exponents}
    \limsup_{k\to\infty}k\left(\alpha-\alpha_k\right)=0.
\end{align}
\end{Lemma}
\begin{proof}
Assume without loss of generality that $x_0=0$.  Take $\varepsilon_0$ and $\rho$ given by Lemma \ref{Lemma_first_iteration}, depending on $\theta$, which will be fixed soon.

Define the nondecreasing sequence
\begin{align*}
    \alpha_k:=\min_{i=0, ..., N}\left\{\alpha_0^-,\min_{x\in B_{\rho^k}}\,\left(\frac{1}{1+\beta_i(x)}\right)\right\},
\end{align*}
which converges to the number
\begin{align*}
    \alpha:=\min_{i=0, ..., N}\left\{\alpha_0^-,\frac{1}{1+\beta_i(0)}\right\}.
\end{align*}
Note that, by \ref{assump_continuity_stronger},
\begin{align*}
    k\left(\frac{1}{1+\beta_i(0)}-\frac{1}{1+\max_{x\in B_{\rho^k}}\beta_i(x)}\right)\leq& k\left(\max_{x\in B_{\rho^k}}\beta_i(x)-\beta_i(0)\right)\\
    \leq & k\omega(\rho^k).
\end{align*}
Therefore, considering all possible cases, we can easily check that
\begin{align*}
    0\leq k(\alpha-\alpha_k)\leq k\omega\left(\rho^k\right),
\end{align*}
with
\begin{align*}
    \limsup_{k\to\infty} k\omega(\rho^k)=0.
\end{align*}
To prove \eqref{Estimate_approximation_geometric_iteration_freeboundary2} and \eqref{Estimate_coefficients_geometric_iteration_freeboundary2}, we will proceed by induction.

Let $\ell_0\equiv 0$ and $\ell_1$ be given by Lemma \ref{Lemma_first_iteration}. Then \eqref{Estimate_approximation_geometric_iteration_freeboundary2} and \eqref{Estimate_coefficients_geometric_iteration_freeboundary2} hold for $k=1$ by Lemma \ref{Lemma_first_iteration}.

Assume that  \eqref{Estimate_approximation_geometric_iteration_freeboundary2} and \eqref{Estimate_coefficients_geometric_iteration_freeboundary2} hold up to $k$. Define
\begin{align*}
    v_k(x)=\frac{(u-\ell_k)}{\rho^{k(1+\alpha_k)}}(\rho^kx).
\end{align*}
Then $\left\|v_k\right\|_{L^\infty(B_1)}\leq 1$ and $v_k$ is a viscosity subsolution to
\begin{align*}
    \min_{i=0, ..., N}&\Bigg\{ \rho^{k\alpha_k \beta_i(\rho^kx)}\left|Dv_k(x)+\rho^{-k\alpha_k}b_k\right|^{\beta_i(\rho^kx)} F_k\left(D^2 v_k\right)\Bigg\}\leq \rho^{k(1-\alpha_k)}\varepsilon_0,
 \end{align*}
where
\begin{align*}
    F_k(M):=\rho^{k(1-\alpha_k)}F\left(\rho^{k(\alpha_k-1)}M\right).
\end{align*}
This implies the following estimate
\begin{align*}
    \min_{i=0, ..., N}\Bigg\{&\left|Dv_k(x)+\rho^{-k\alpha_k}b_k\right|^{\beta_i(\rho^kx)} F_k\left(D^2 v_k\right)\Bigg\}\\
     \leq & \max_{i=0, ..., N}\left\{\rho^{k(1-\alpha_k)-k\alpha_k \beta_i(\rho^kx)} \right\}\varepsilon_0\leq \varepsilon_0,
 \end{align*}
where the last inequality follows from the definition of $\alpha_k$. Calling \linebreak $p_k:=\rho^{-k\alpha_k}b_k$ and $\beta_i^k(x):=\beta_i(\rho^k x)$, we get that $v_k$ is a subsolution to
\begin{align*}
    \min_{i=0, ..., N}\Bigg\{&\left|Dv_k(x)+p_k\right|^{\beta_i^k(x)} F_k\left(D^2 v_k\right)\Bigg\}\leq \varepsilon_0.
\end{align*}
Similarly, we prove that $v_k$ is a viscosity supersolution to
\begin{align*}
    \max_{i=0, ..., N}\Bigg\{&\left|Dv_k(x)+p_k\right|^{\beta_i^k(x)} F_k\left(D^2 v_k\right)\Bigg\}\geq -\varepsilon_0.
\end{align*}
Note that $\beta_i^k$ still satisfy assumption \ref{assump_continuity_stronger}.

Hence, we can use Lemma \ref{Lemma_first_iteration} to guarantee the existence of a linear function $\overline{\ell}(x)=\overline{a}+\overline{p}\cdot x$ such that
\begin{align*}
    \sup_{B_\rho}|v_k-\overline{\ell}|\leq \rho^{1+\theta},
\end{align*}
where $\theta=\frac{\alpha+\alpha_0}{2}<\alpha_0$ and the coefficients satisfy
\begin{align*}
    &\overline{a}=h(0),\\
    &\overline{p}=Dh(0),
\end{align*}
where $h$ is a viscosity solution to $G(D^2h)=0$ and $G$ has the same ellipticity constants as $F$. Hence, as a straightforward application of ellipticity, $h$ has interior $C^{1, \alpha_0}$ estimates which imply universal bounds on the coefficients of $\overline{\ell}$.

Rescalling back to the unit ball, we get
\begin{align*}
    &\sup_{x\in B_\rho}\left|\frac{u-\ell_k}{\rho^{k(1+\alpha_k)}}(\rho^kx) -\overline{\ell}(x)\right|\leq \rho^{1+\theta}\\
    \iff&\sup_{y\in B_{\rho^{k+1}}}\left|u(y)-\ell_k(y)-\rho^{k(1+\alpha_k)}\overline{\ell}(\rho^{-k}y) \right|\leq \rho^{1+\theta}\rho^{k(1+\alpha_k)}\\
    \iff&\sup_{y\in B_{\rho^{k+1}}}\left| u(y)-\ell_{k+1}(y) \right|\leq \rho^{1+\theta}\rho^{k(\alpha_k-\alpha_{k+1})}\rho^{k(1+\alpha_{k+1})},
\end{align*}
where
\begin{align*}
    \ell_{k+1}(y)-\ell_k(y)=&(a_{k+1}-a_k)+(p_{k+1}-p_k)\cdot y\\
    =&\rho^{k(1+\alpha_k)}h(0)+\rho^{k\alpha_k}Dh(0)\cdot y.
\end{align*}
Because of \eqref{Equation_bound_continuity_modulus}, we have
\begin{align*}
    \rho^{k(\alpha_k-\alpha_{k+1})}\leq\rho^{-k\omega(\rho^k)}\leq \rho^\frac{\alpha-\alpha_0}{2},
\end{align*}
hence, we can further estimate
\begin{align*}
    \rho^{1+\theta}\rho^{k(\alpha_k-\alpha_{k+1})}\rho^{k(1+\alpha_{k+1})}\leq\rho^{\theta-\alpha_{k+1}}\rho^\frac{\alpha-\alpha_0}{2}\rho^{(k+1)(1+\alpha_{k+1})}
\end{align*}
and since $\theta=\frac{\alpha+\alpha_0}{2}$, we can write
\begin{align*}
    \theta-\alpha_{k+1}+\frac{\alpha-\alpha_0}{2}=\alpha-\alpha_{k+1}\geq 0.
\end{align*}
Combining all these inequalities, we can finally estimate
\begin{align*}
\sup_{y\in B_{\rho^{k+1}}}\left| u(y)-\ell_{k+1}(y) \right|\leq \rho^{(k+1)(1+\alpha_{k+1})}
\end{align*}
which proves \eqref{Estimate_approximation_geometric_iteration_freeboundary2} and since $|h(0)|\leq C$, $|Dh(0)|\leq C$, estimate \eqref{Estimate_coefficients_geometric_iteration_freeboundary2} follows immediately aswell. This concludes the proof for the case $x_0=0$. A standard translation locates this argument at any point $x_0\in B_{1/2}$.
\end{proof}

Theorem \ref{Theorem_pointwise_regularity} follows from Lemma \ref{Lemma_geometric_iterations_pointwise} together with Proposition \ref{Proposition_Holder_characterization}.

\bigskip

{\small \noindent{\bf Acknowledgments.}  The author is very grateful to E. A. Pimentel and J. M. Urbano for pointing out this problem and for the fruitful discussions they had together. DJ was supported by FCT, Portugal, through scholarship PD/BD/150354/2019, under POCH funds, co-financed by the European Social Fund and Portuguese National Funds from MEC, Portugal and by the Centre for Mathematics of the University of Coimbra - UIDB/00324/2020, funded by the Portuguese Government through FCT/MCTES.
}

\smallskip

{\small \noindent{\bf Declaration.}  The author declares that he has no conflict of interest.}


\begin{thebibliography}{99}



\bibitem{Acerbi-Mingione2001} E. Acerbi and G. Mingione, 
\textit{Regularity results for a class of functionals with non-standard growth}, 
 Arch.
Ration. Mech. Anal. 156 (2001) 121–140.



\bibitem{Birindelli-Demengel2004} I. Birindelli and F. Demengel, 
\textit{Comparison principle and Liouville type results for singular fully nonlinear operators}, 
Ann. Fac. Sci. Toulouse Math. (6) 13 (2) (2004), 261–287.

\bibitem{Birindelli-Demengel2007} I. Birindelli and F. Demengel, 
\textit{Eigenvalue, maximum principle and regularity for fully non linear homogeneous operators},  Commun. Pure Appl. Anal. 6 (2) (2007), 335–366.

\bibitem{Birindelli-Demengel2009} I. Birindelli and F. Demengel, 
\textit{Eigenvalue and Dirichlet problem for fully-nonlinear operators in non-smooth domains},  J. Math. Anal. Appl. 352 (2) (2009), 822–835.


\bibitem{Blomgren-Chan-Mulet-Vesse-Wan1999} P. Blomgren, T. Chan, P. Mulet, L. Vese, W. Wan, 
\textit{Variational PDE models and methods for image
processing}, in: Numerical Analysis 1999, Dundee, in: Chapman and Hall/CRC Res. Notes Math.,
vol. 420, Chapman and Hall/CRC, Boca Raton, FL, (1999), 43–67.



\bibitem{Borsuk2010} M. Borsuk, 
\textit{Transmission problems for elliptic second-order equations in non-smooth domains}, 
Frontiers in Mathematics. Birkh\"auser/Springer Basel AG, Basel, 2010.



\bibitem{Bronzi-Pimentel-Rampasso-Teixeira2020} A. C. Bronzi, E. A. Pimentel, G. Rampasso and E. V. Teixeira, 
\textit{Regularity of solutions to a class of variable–exponent fully nonlinear elliptic equations}, 
J. Funct. Anal. 279 (12) (2020), 108781, 31pp.



\bibitem{Caffarelli-Cabre1995} L.A. Caffarelli and X. Cabr\'e, 
\textit{Fully Nonlinear Elliptic Equations}, 
American Mathematical Society Colloquium Publications, 43. American Mathematical Society, Providence, RI, 1995.


\bibitem{Crandall-Ishii-Lions1992} M. G. Crandall, H. Ishii, and P.L. Lions,  
\textit{User’s guide to viscosity solutions of second order partial differential equations}, Bull. Amer. Math. Soc.
(N.S.), 27 (1) (1992), 1-67.


\bibitem{Filippis2020} C. De Filippis, \textit{Regularity for solutions of fully nonlinear elliptic equations with nonhomogeneous degeneracy}, Proc. Roy. Soc. Edinburgh Sect. A 151 (2021), no. 1, 110–132.


\bibitem{Huaroto-Pimentel-Rampasso-Swiech2020} G. Huaroto, E. A. Pimentel, G. C. Rampasso and A. \'Swi\k{e}ch,  
\textit{A fully nonlinear degenerate free transmission
problem}, arXiv preprint arXiv:2008.06917v2.


\bibitem{Imbert-Silvestre2013}C. Imbert and L. Silvestre, 
\textit{$C^{1,\alpha}$ regularity of solutions of some degenerate
fully non-linear elliptic equations}, Adv. Math. 233 (2013), 196–206.

\bibitem{Imbert-Silvestre2016}C. Imbert and L. Silvestre, 
\textit{Estimates on elliptic equations that hold only
where the gradient is large}, . J. Eur. Math. Soc. (JEMS), 18 (6) (2016), 1321–1338.
 

\bibitem{Lian-Wang-Zhang2020} Y. Lian, L. Wang and K. Zhang,  
\textit{Pointwise Regularity for Fully Nonlinear Elliptic Equations in General Forms}, arXiv preprint arXiv:2012.00324v1.



\end{thebibliography}
\end{document}